\documentclass{amsart}
\usepackage{graphicx}
\usepackage{FancyMath}

\title{The Heegaard Floer d-invariant for more rational homology spheres}
\author{Isabella Khan}
\date{6 May 2026}

\begin{document}

\begin{abstract}
The Heegaard Floer $d$-invariant for a rational homology sphere $Y$ and spin$^c$-structure $\s$ is defined as the minimal absolute grading of a generator of $HF^+(Y; \s)$. In 2005, N\'emethi used lattice homology to compute the d-invariant for a particular class of negative-definite plumbed rational homology spheres, and conjectured that his formula should hold for all negative-definite plumbed rational homology spheres. In this paper, we use Zemke's isomorphism between lattice and Heegaard Floer homology to prove N\'emethi's conjecture.
\end{abstract}
	\maketitle\blfootnote{The author was supported by the Simons Collaboration for New Structures in Low-dimensional Topology.}

\maketitle

\section{Introduction}

    Heegaard Floer homology is a powerful 3-manifold invariant first constructed by Ozsv\'ath and Sz\'abo in~\cite{OShol1}. To any closed 3-manifold $Y$, the Heegaard Floer package associates a certain type of module $\HF^{\circ}(Y)$; while the type of module varies based on the particular flavor of Heegaard Floer homology which is being computed, each flavor splits over the spin$^c$-structures of $Y$, so that $\HF^{\circ}(Y)$ is a direct sum of submodules $\HF^{\circ}(Y; \s)$ for each $\s \in \spin^c(Y)$. When the spin$^c$-structure $\s$ is torsion, $\HF^{\circ}(Y; \s)$ admits an absolute $\Q$-grading, $\grt$. In this paper, we concentrate on rational homology spheres, on which all spin$^c$-structures are torsion, so that we can define $\grt$ on $\HF^{\circ}(Y; \s)$ for each spin$^c$-structure of $Y$. In these cases, we have a fundamental invariant of the 3-manifold $Y$ called the \emph{Heegaard Floer $d$-invariant}. For each (torsion) $\s \in \spin^c(Y)$, $d(Y, \s)$ is the minimal absolute grading of a non-torsion element of $HF^+(Y;\s)$. We can also define $d(Y, \s)$ as the maximal grading of a non-torsion element of $\HF^-(Y; \s)$, which is the approach we will use in this paper.

    In 2005, N\'emethi used lattice homology to prove in~\cite{Nem05} that for certain rational homology spheres $Y$ -- those which admitted a particular kind of negative-definite plumbing called an \emph{almost rational plumbing} — the $d$-invariant can be computed for each $\s \in \spin^c(Y)$, as:
    \begin{equation}\label{Nem1}
        d(Y, \s) = \max_{k \in \s} \frac{k^2 + |V|}{4}
    \end{equation}
    where $|V|$ is the number of vertices of the plumbing graph, $k$ ranges across the characteristic vectors of the associated 4-manifold $X$ which represent $\s$, and the ``$^2$'' in $k^2$ indicates the intersection form of $X$. See Theorem 8.3 and Proposition 4.7 of~\cite{Nem05} for an exact statement of this result.

    In~\cite{Nem07}, N\'emethi conjectured further that~\eqref{Nem1} should hold for all negative-definite plumbed rational homology spheres. In this paper, we verify N\'emethi's conjecture. The main result is:
    \begin{theorem}\label{main}
        Let $Y$ be a negative-definite plumbed 3-manifold which is a rational homology sphere, and let $[k]$ be a spin$^c$ structure for $Y$. Then
       \begin{equation}\label{main1}
           d(Y, [k]) = \max_{k' \in [k]} \frac{(k')^2 + s}{4},
       \end{equation}  
        where $s$ denotes the number of vertices of the plumbing graph for $Y$, and $(k')^2 = (k')^T M (k')$, where $M$ is the plumbing matrix corresponding to $Y$. 
    \end{theorem}
    \noindent This result is interesting for several reasons. First, the $d$-invariant is relevant on its own merits as a topological invariant of rational homology spheres. Also, given the work of Kutluhan-Lee-Taubes in~\cite{KLT} (and its sequels) identifying the Heegaard Floer invariants with Seiberg-Witten invariants, the $d$-invariant should coincide with the Fr{\o}yshov $h$-invariant, so Theorem~\ref{main} gives new information about the Fr{\o}yshov invariant as well. 

    The author's particular interest in the $d$-invariant stems from quantum topology, specifically, the conjectural connections between the Gukov-Pei-Putrov-Vafa (GPPV) invariant and Floer theory. The GPPV invariant associates to a 3-manifold $Y$, with spin$^c$-structure $[k]$, a $q$-series $\hat{Z}_{[k]}$, which Murakami showed in~\cite{Mur25} converges to the Witten-Reshitikhin-Turaev invariants in a certain sense. The $\hat{Z}$-invariant was computed Gukov-Manolescu for  Brieskorn spheres in~\cite{GM19}. In~\cite{AJK}, Akhmechet-Johnson-Krushkal used N\'emethi's graded root construction from lattice homology to give an algorithm for computing $\hat{Z}$ for negative-definite plumbed 3-manifolds, and in~\cite{Li24}, Liles used this algorithm to compute the $\hat{Z}$ for Brieskorn spheres. 

    Given the existing connections between lattice homology and the $\hat{Z}$-invariant, as well as Zemke's isomorphism between Heegaard Floer and lattice homology from~\cite{Zem21}, it seems plausible to expect some meaningful connection between Heegaard Floer theory and the $\hat{Z}$-series. The $d$-invariant could provide one basis for such a connection: the minimal $q$-degree of terms in the $\hat{Z}$-series is a rational number usually denoted $\Delta$ in the literature, which was conjectured in~\cite{AJK} to be equal to the $d$-invariant for rational homology spheres. In~\cite{HNS}, Harichurn-N\'emethi-Svoboda show that this conjecture is false; however, the broader question of whether, and how, the $d$-invariant and $\Delta$ might in general be related, remains open. The author hopes that this new, calculation of the $d$-invariant might provide some insight into this question.

    This paper is structured as follows. In Section~\ref{plu}, we give review terminology for plumbed 3-manifolds, and establish notation which will be used in the rest of the paper. In Section~\ref{hf}, we give a brief introduction to the Heegaard Floer invariants, and in Section~\ref{ls00}, we give the definition of Manolescu-Ozsv\'ath's link surgery complex of~\cite{MO}, the reformulation of Heegaard Floer homology which will be used in proving Theorem~\ref{main}. In Section~\ref{lh}, we define lattice homology and discuss the various results connecting lattice homology with Heegaard Floer homology, from~\cite{OSSspec} and~\cite{Zem21}. Finally, in Section~\ref{ver}, we give a proof of Theorem~\ref{main}.

    \begin{ackno}
       \emph{I would like to thank Shimal Harichurn for suggesting this problem and Peter Ozsv\'ath for suggesting a possible approach. I am also grateful to  Andr\'as N\'emethi, Ian Zemke, and Tom Mrowka for many helpful comments during the preparation of this paper.}
    \end{ackno}

    \section{Plumbed 3-manifolds}\label{plu}

        All manifolds considered in this paper will be negative-definite plumbed 3-manifolds. In this section, we therefore provide a brief review of definitions and terminology for these manifolds. A more complete exposition of this topic can be found in any number of sources, for instance~\cite{GomStip}.
        
        The data for a plumbed manifold is a \emph{plumbing graph} $G$ which has finitely many vertices and no cycles. In this paper, we will also assume all graphs are connected. We write $V(G) = \{v_1, \ldots, v_s\}$ to denote the vertices of $G$, and $E(G)$ to denote the edges of $G$; we say that $[v_i, v_j] \in E(G)$ if and only if $v_i$, $v_j$ are connected by an edge of $G$. Each vertex $v_i \in V(G)$ is decorated with an integer weight $m_i$; we write $\vec{m} = (m_1, \ldots, m_s)$ for this weight vector. Later, we will usually suppress the $G$ from the notation for $V(G)$ and $E(G)$. The plumbing graph and weight vector naturally determine a framed link of unknots $L = \bigcup_{i = 1}^s K_i \subseteq S^3$, with framing given by $\vec{m}$, which we call the \emph{plumbing link}.  Define the \emph{plumbing matrix} $M = (a_{ij})$ corresponding to $G$ as
        \[
           a_{ij} = \begin{cases}
                m_i & \text{ if } i = j; \\
                 1 & \text{ if } i \neq j \text{ and } [v_i, v_j] \in E(G); \\
                 0 & \text{ otherwise}
                \end{cases}
        \]
        To obtain a 3-manifold $Y$ from this set-up, we do $\vec{m}$-framed Dehn surgery along $L \subseteq S^3$. The corresponding 4-manifold $X$ is constructed by attaching an $m_i$-framed $2$-handle to $B^4$ along each component $K_i$ of $L$. Then we have $Y = \del X$.
    
        Next, we discuss characteristic vectors and spin$^c$ structures. There exists a canonical basis $\{[v_i]\}_{i = 1}^s$ for $H_2(X; \Q)$. When $b_1(Y) = 0$, we have $H^2(X; \Q) \cong \Hom(H_2(X; \Q), \Q)$, so we have a dual basis $\{[v_i]^*\}_{i = 1}^s$ for $H^2(X; \Q)$. In this case, we write
        \[
           k \cdot x = \sum_{i = 1}^s k_i x_i,
        \]
        for each $k = \sum_{i = 1}^s k_i [v_i]^* \in H^2(X; \Q)$ and $x = \sum_{i = 1}^s x_i [v_i] \in H_2(X; \Q)$. We write $(x, y) = x^T M y$ in this basis, so that $M$ determines the intersection form of $X$. We will often write $x^2$ for $(x, x)$, as we do in the statement of Theorem~\ref{main}. The 3-manifold $Y$ is a \emph{negative definite} plumbed 3-manifold precisely when $M$ determines a negative definite intersection form on $X$. The \emph{characteristic vectors of $X$} are defined as 
        \[
            \Char(X) := \{ k \in H^2(X; \Q): k\cdot x + (x, x) = 0 \text{ for each } x \in H_2(X; \Q)\}.
        \]
        Notice that in the canonical basis $\{[v_i]^*\}$ for $H^2(X; \Q)$, we have
        \[  
            \Char(X) \cong \vec{m} + 2\Z^s.
        \]
        Recall that taking the first Chern class defines a bijection $c_1: \spin^c(X) \to \Char(X)$. When we restrict our spin$^c$ structures to $Y$, this determines an isomorphism
        \begin{equation}\label{plu1}
            \spin^c(Y) \cong \frac{\Char(X)}{2 M \Z^s} = \frac{\vec{m} + 2 \Z^s}{2 M\Z^s}
        \end{equation}
        in the canonical basis; that is $k, k' \in \Char(X)$ are representatives of the same element of $\spin^c(Y)$ if and only if $k - k' \in 2 M \Z^s$ in the canonical basis.

    \section{Heegaard Floer homology}\label{hf}

        Heegaard Floer homology is an invariant of closed 3-manifolds with wide applications to problems in low-dimensional topology. The goal in this section is not to give a full introduction either to the construction or the applications of this invariant, but to provide just enough information to understand the sections that follow and the proof of Theorem~\ref{main}. The reader interested in a more thorough treatment is directed to~\cite{OShol1},~\cite{OSAbs}, or~\cite{OSlect}, among many excellent sources.

        \subsection{Definitions}

            The Heegaard Floer homology of a closed 3-manifold $Y$ is constructed as follows. A \emph{Heegaard diagram} $(\Sigma, \al, \be)$ consists of a genus $g$ surface $\Sigma$ and families $\al = \{\alpha_1, \ldots, \alpha_g\}$ and $\be = \{\beta_1, \ldots, \beta_g\}$ of closed curves in $\Sigma$, such that $\al \cup \be$ generates $H_1(\Sigma)$. Let $U_{\al}$ denote the handlebody with boundary $\Sigma$ determined by attaching 2-handles along the $\alpha$-circles, and $U_{\be}$ denote the handlebody, also with boundary $\Sigma$, determined by attaching 2-handles along the $\beta$-circles. We also choose a basepoint $z \in \Sigma$ in the complement of the $\alpha$- and $\beta$-arcs. Then $(\Sigma, \al, \be, z)$ is a \emph{Heegaard splitting} of $Y$ if $Y = U_{\al} \cup_{\Sigma} U_{\be}$; it is a well known fact that every closed 3-manifold $Y$ admits a Heegaard splitting. 

            Given a Heegaard splitting $(\Sigma, \al, \be, z)$ for $Y$, let $\Tt_{\al}$ and $\Tt_{\be}$ denote the $g$-dimensional Lagrangian tori in the $2g$-dimensional symplectic manifold $\Sym^g(\Sigma)$. Then $\CF^{\circ}(Y)$ is generated by intersection points $\x \in \Tt_{\al} \cap \Tt_{\be}$. (The ``$\circ$'' in $\CF^{\circ}$ denotes one of the flavors of Heegaard Floer homology, which we describe in slightly more detail in the next paragraph.) Letting $\pi_2(\x, \y)$ denote the set of holomorphic disks in $\Sigma$ with boundary along $\Tt_{\al} \cup \Tt_{\be}$ and corners at $\x$ and $\y$, the differential counts rigid moduli spaces of holomorphic disks $\pi_2(\x, \y)$. More specifically, Heegaard Floer homology is equipped with a relative $\Z$-grading $\mu$, called the \emph{Maslov grading}, defined by assigning a $\Z$-value to each element of $\pi_2(\x, \y)$ and extended as a relative grading to generators as
            \[  
                \mu(\y) - \mu(\x) = \mu(\phi)
            \]
            for each $\phi \in \pi_2(\x, \y)$.

            There are several flavors of Heegaard Floer homology -- $\widehat{\HF}$, $\HF^+$, $\HF^-$, $\HF^{\infty}$, and $\HFh^-$ -- which are determined by the ground ring and number of base-points in $\Sigma \smallsetminus (\bigcup_i \alpha_i) \cup ( \bigcup_i \beta_i )$ involved in the construction. Here and in what follows, we will be using $\Z_2$-coefficients, so $\F$ will always denote $\Z_2$. All flavors of Heegaard Floer homology are generated by the intersections $\x \in \Tt_{\al} \cap \Tt_{\be}$ in $\Sym^g(\Sigma)$. Having said this, $\CF^{\infty}(Y)$ is generated over $\F[U, U^{-1}]$, so that generators are written as $[\x, i]$, where $\x \in \Tt_{\al} \cap \Tt_{\be}$, $i \in \Z$, and the $U$-action raises $i$ by $1$ and $U^{-1}$-action drops $i$ by 1. The complex $\CF^{-}(Y)$ is generated over $\F[U]$ by $[\x, i]$ with $i < 0$ and the stipulation that $U\cdot [\x, 0] = 0$ for each $\x$; $\CF^+(Y)$ is generated over $\F[U]$ by $[\x, i]$ with $i > 0$, and $\widehat{\HF}(Y)$ is generated over $\Z$ by $[\x, 0]$. Finally, $\CFh^-(Y)$ is the completion of $\CF^-(Y)$ over ground ring $\F[[U]]$. The reader interested in a more thorough discussion of the basepoints how they affect the differentials in these various cases is directed to e.g.~\cite{OSlect}.

            All flavors of Heegaard Floer homology split over spin$^c$-structures of $Y$, as
            \[
                \HF^{\circ}(Y) = \bigoplus_{[k] \in \spin^c(Y)} \HF(Y; [k]).
            \]
            In this paper, we will primarily be concerned with the completed version $\HFh^-(Y)$, which appears in Zemke's isomorphism with lattice homology. 

        \subsection{Cobordisms and gradings}\label{cogr}
            
            If our 3-manifold $Y$ is equipped with a torsion spin$^c$-structure $\s$, then there is also an absolute $\Q$-grading $\widetilde{\gr}$ which lifts the relative $\Z$-grading $\mu$ in the sense that for any pair of generators $\x, \y$ and $\phi \in \pi_2(\x, \y)$, 
            \[
                \grt(\y) - \grt(\x) = \mu(\phi).
            \]
            This grading is constructed manually on $\HF^{\infty}(S^3, \s_0)$ in~\cite{OS01} (where $\s_0$ denotes the canonical spin$^c$-structure on $S^3$), and has the property that the minimal grading element of $HF^+(Y; \s)$ has $\grt(\x_0) = 0$. For an arbitrary closed 3-manifold $Y$ with torsion spin$^c$-structure $\s$, the absolute grading on $\HF^{\infty}(Y; \s)$ is constructed using the grading on $\HF^{\infty}(S^3)$, together with \emph{cobordism maps} for Heegaard Floer homology. In order to discuss the absolute grading in more detail in what follows, we now give an overview of the construction of these maps.  In this paper, we will only be interested in cobordisms between $S^3$ and a closed manifold $Y$, so we describe only those maps here. 

            Any 3-manifold $Y$ can be obtained via Dehn surgery along some framed link in $S^3$. For a link $L = \bigcup_{i = 1}^s K_i \subseteq S^3$, \emph{a bouquet $B(L)$ for the link $L$} is a 1-complex in $S^3$ which is the union of $L$ with a collection of paths joining each component to a particular basepoint in $S^3$. Let $L = \bigcup_{i = 1}^s K_i$ be a link in $S^3$ with framing $\vec{m} = (m_1, \ldots, m_s)$, such that $Y$ is obtained by $\vec{m}$-framed Dehn surgery along $L$. A \emph{Heegaard triple subordinate to $B(L)$} is a triple $(\Sigma, \al, \be, \ga, z)$, with $\al = \{\alpha_i\}_{i = 1}^g, \be = \{\beta_i\}_{ i = 1}^g,$ and $\ga = \{\gamma_i\}_{i = 1}^g$ closed curves in $\Sigma$ and $z$ in the complement of the $\alpha$-, $\beta$-, and $\gamma$-curves such that:
            \begin{itemize} 
                \item $(\Sigma, \al, \be)$ is a Heegaard diagram for the complement of the bouquet $B(L)$;

                \item $\gamma_{s + 1}, \ldots \gamma_g$ are small isotopies of the $\beta_{s + 1},\ldots, \beta_g$;

                \item For each $1 \leq i \leq s$, each $\beta_i$ and $\gamma_i$ lie in the boundary of a neighborhood of $K_i$;

                \item For each $1 \leq i \leq s$, the $\beta_i$ are meridians for the $K_i$, and are disjoint from each $\gamma_j$ with $i \neq j$ and meet $\gamma_i$ in a single transverse intersection;

                \item For each $1 \leq i \leq s$, $\gamma_i$ corresponds to the framing of $K_i$ under the natural identification. 
            \end{itemize}
            It is shown in~\cite{OS01} Proposition 4.3 that a Heegaard triple of this form determines a 4-manifold $X_{\alpha\beta\gamma}$ with boundary components
            \begin{align*}
                Y_{\alpha\beta} &= S^3 \\
                Y_{\beta\gamma} &= \#^{g - s} (S^1 \times S^2)\\
                Y_{\alpha \gamma} &= Y
            \end{align*}
            and that after filling $Y_{\beta\gamma}$ with the boundary connected sum $\#^{g - s}(S^1 \times B^3)$, we obtain the standard cobordism $W$ between $S^3$ and $Y$. 

            Let $\s$ be a spin$^c$-structure on $Y$. Write $Y_1 = \#^{g - s} (S^1 \times S^2)$ and let $\s_1$ denote the torsion spin$^c$-structure of $Y_1$. Recall from Section 2.4 of~\cite{OS01} that there is a well-defined top-dimensional homology group of $\widehat{\HF}(Y_1; \s_1) = \Z[U]$, and let $\Theta \in \widehat{\HF}(Y_1, \s_1)$ be one of its generators. (Here, we are using the relative $\Z$-grading $\mu$ to determine the degree.) Let $\mathfrak{t} \in \spin^c(X_{\alpha \beta \gamma})$ be a spin$^c$-structure which restricts to the canonical spin$^c$ structure $\s_0$ on $Y_{\alpha \beta} = S^3$, to $\s_1$ on $Y_{\beta \gamma} = Y_1$, and to $\s$ on $Y$. Then we have a natural map
            \begin{equation}\label{hee1}
                F_{W, \s}^{\circ}: \CF^{\circ}(S^3; \s_0) \to \CF^{\circ}(Y, \s)
            \end{equation}
            for each flavor of Heegaard Floer homology, such that for $\x \in \CF^{\circ}(S^3; \s_0)$, $F_{W, \s}^{\circ}(\x)$ is a sum of $\y \in \CF^{\circ}(Y, \s)$ such that there exist rigid holomorphic triangles in $\pi_2(\x, \Theta, \y)$ (i.e. triangles with boundary along $\Tt_{\al}, \Tt_{\be}$ and $\Tt_{\ga}$ in $\Sym^g(\Sigma)$, with corners at $\x, \Theta$ and $\y$ and $\mu(\phi) = 0$), with coefficients in the ground ring. 

            Suppose we have a torsion spin$^c$ structure $\s$ on $Y$. Then the absolute grading $\grt$ is defined on generators $\y\in \CF^{\circ}(Y; \s)$ according to the rule that for all notation as above $\x$ a generator of $\CF^{\circ}(S^3; \s_0)$, and $\phi \in \pi_2(\x, \Theta, \y)$, 
            \begin{equation}\label{hee3}
                \grt(\y) = \mu(\phi) + n_z(\phi) + \frac{c_1(\mathfrak{t}) - 2 \chi(W) - 3 \sigma(W)}{4}.
            \end{equation}
            Moreover, $\grt$ has the property that for cobordism maps as in~\eqref{hee1} and $\x$ a generator of $\CF^{\circ}(S^3; \s_0)$,
            \begin{equation}\label{hee2}
                \grt(F_{W, \s}(\x)) - \grt(\x) = \frac{ c_1(\mathfrak{t}) - 2 \chi(W) - 3 \sigma(W)}{4}.
            \end{equation}
            The cobordism maps will not arise explicitly in what follows. However, see Remark~\ref{ls8}, below, for a discussion of how cobordism maps are involved in the construction of an absolute grading for the link surgery complex, which is why they are included here.

            \section{The link surgery formula}\label{ls00}

            In the proof of Theorem~\ref{main}, we will use an identification between a reformulation of Heegaard Floer homology due to Manolescu and Ozsv\'ath, called the \emph{link surgery complex}, and lattice homology. This identification will be discussed below, in Section~\ref{isom}, and in order to understand it, we need to give the definition of the link surgery complex here. 

             The link surgery complex is a reformulation of Heegaard Floer homology defined in the case of plumbed 3-manifolds. For more details beyond the treatment given here,  see~\cite{MO}, in which the link surgery complex was originally defined, or~\cite{OSSspec} or~\cite{IZbord} for slightly later discussions.

             We first need to give the definitions of several types of Heegaard diagrams for links. This treatment follows Section 3.1 of~\cite{MO}. First, a \emph{multipointed Heegaard diagram} is $\Hh = (\Sigma, \al, \be, \ga, \w, \z)$ such that
             \begin{itemize}
                 \item (Surface) $\Sigma$ is a genus $g$ surface;
                 \item (Circles) For some $k \geq 0$, $\al, \be$ are each collections of $(g + k - 1)$ simple closed curves spanning a $g$-dimensional subspace of $H_1(\Sigma; \Z)$;
                 \item (Basepoints) Let $\{A_i\}_{i = 1}^k$ and $\{B_i\}_{i = 1}^k$ be the connected components of $\Sigma \smallsetminus \bigcup_{i = 1}^{g +k - 1} \alpha_i$ and $\Sigma \smallsetminus \bigcup_{i = 1}^{g + k - 1} \beta_i$, respectively. Then for some $m \leq k$, the basepoints are $\w = (w_1, \ldots, w_k)$ and $\z = (z_1, \ldots, z_m)$, with the property that there is some permutation $\sigma$ of $\{1, \ldots, m\}$ such that
                 \begin{align*}
                    w_i \in A_i \cap B_i & \text{ for each } 1 \leq i \leq k \\
                    z_i \in A_i \cap B_{\sigma(i)} & \text{ for each } 1 \leq i \leq m
                 \end{align*}
             \end{itemize}
             Any multi-pointed diagram $\Hh$ specifies an oriented link $\vec{L} = \bigcup_{i = 1}^s \vec{K}_i \subseteq S^3$ according to the following procedure. Let $U_{\al}$ and $U_{\be}$ denote the $\alpha$-and $\beta$-handlebodies discussed at the beginning of Section~\ref{hf}. Then for each $1 \leq i \leq m$, join $w_i$ to $z_i$ via an arc in $A_i$, which we push into $U_{\al}$, and join $z_i$ to $w_{\sigma(i)}$ via an arc in $B_i$, which we push into $U_{\be}$. Now we have a link $\vec{L}$ with $s \leq m$ components. We call the extra $w_{m+1}, \ldots, w_{k}$ \emph{free basepoints},
             
             A \emph{Heegaard Diagram for the link $L$} is a multipointed Heegaard diagram $\Hh = (\Sigma, \al, \be, \w, \z)$ which specifies $L$ according to the procedure above. We say that $\Hh$ is \emph{link minimal} if $m = s$, i.e. each link component has exactly 2 basepoints. We say that $\Hh$ is \emph{minimally pointed} if it is link minimal and $k = m = s$, i.e. there are no free basepoints. For a link $L$ of unknots, it is always possible to find a minimally pointed Heegaard diagram for $L$ (as shown in Section 3.1 of~\cite{MO}.

            The link surgery complex is constructed in the following way. Let $Y$ be a plumbed 3-manifold with corresponding plumbing link $L = \bigcup_{i = 1}^{\ell} K_i$. Write $\Lambda = (m_1, \ldots, m_{\ell}) \in \Z^{\ell}$ for the framing of $L$, which is also the vector of integer weights of the original plumbing graph, and let $\mathcal{H}^L = (\Sigma, \al, \be, \w, \z)$ denote a link minimal Heegaard diagram for the link $L$. Define the lattice
            \[
                \mathbb{H}(L) = \prod_{i = 1}^{\ell} \left( \frac{\mathrm{lk}(L \smallsetminus K_i, L)}{2} + \Z\right).
             \]
            The link surgery complex is defined as the direct sum over sublinks $M$ of $L$:
             \[
                \Cc_{\Lambda}(L) = \bigoplus_{M \subseteq L} \prod_{\s \in \mathbb{H}(L)} \Mm^-(\mathcal{H}^{L \smallsetminus M}, \psi^{ M}(s)),
            \]
            where $\Hh^{L \smallsetminus M}$ denotes the Heegaard diagram obtained from $\Hh^{L}$ by deleting the basepoint $z_i$ for each component $K_i$ which is included in M, and $\psi^{M}: \mathbb{H}(L) \to \mathbb{H}(L \smallsetminus M)$ is a restriction map which deletes the components of $\s$ corresponding to components of $M$. The space $\mathfrak{A}^-(\mathcal{H}^{L \smallsetminus M}, \psi^{ M}(s))$ is defined as the $\F[[U_1, \ldots, U_{\ell}]]$-module generated by terms of the form
            \[
              U_1^{i_1} \cdots U_{\ell}^{i_{\ell}} \cdot \x,
             \]
            where $\x \in \mathbb{T}_{\al} \cap \mathbb{T}_{\be}$ is a generator of the link Floer complex, $i_j \geq 0$ for each $j$, we place a further lower bounds on each $i_j$ based on the component $s_j$ of $s$, and the $j$-th component of the Alexander grading of $\x$. For more details on the exact conditions, see e.g. Section of~\cite{MO} or Section 6.1 of~\cite{IZbord}. We will not need the exact conditions in this paper. Note also that in practice, the action of each of the $U_i$ have the same action, so we can view $\C_{\Lambda}(L)$ as a module over $\F[[U]]$, where all $U_i$ act by $U$.

            The complex $\Cc_{\Lambda}(L)$ has a hypercube structure which we will see in Section~\ref{isom} parallels the hypercube structure of the lattice homology complex. Each sublink $M$ of $L$ corresponds to a point $\epsilon \in \{0, 1\}^{\ell}$, where $\epsilon_i = 0$ if $K_i \nsubseteq L$ and 1 if $K_i \subseteq L$.  The hypercube maps $\Phi^{N}$ map from the component of $\Cc_{\Lambda}(L)$ corresponding to $L \smallsetminus M$, to the one corresponding to $L \smallsetminus (M \cup N)$. The map $\Phi^{N}$ contributes to the length-$r$ hypercube map of $\Cc_{\Lambda}(L)$ precisely when $|N| = r$. These maps determine an overall differential $\mathcal{D}$ on $\Cc_{\Lambda}(L)$, which decomposes as a sum
            \begin{equation}\label{ls6}
               \mathcal{D} = \sum_{r = 0}^{\infty} \mathcal{D}^r,
             \end{equation}
            where $\mathcal{D}^r$ is the length-$r$ hypercube map for each $r$. Let $(\Cc_{\Lambda}(L), \D^0)$ denote the chain complex obtained by applying only the length-0 hypercube map on the link surgery complex, and let $H_*(\mathcal{C}_{\Lambda}(L))$ denote the homology of this chain complex. In Section~\ref{isom}, we will discuss Ozsv\'ath, Stipsicz, and Szab\'o's  $\F[[U]]$-module isomorphism between the chain complex $(H_*(\mathcal{C}_{\Lambda}(L)), \D^1)$ and the lattice homology complex. (See Proposition~\ref{isom4} below.)

            Finally, we note the special case of the main result of~\cite{MO} which we need in order to verify Theorem~\ref{main}. This theorem provides a more detailed statement of the fact that the link surgery complex is a reformulation of the Heegaard Floer complex.
            \begin{theorem}\label{ls7}
                [Special case of Theorem 9.6 of~\cite{MO}] Let $Y$ be a 3-manifold obtained by $\Lambda$-surgery along the link $L \subseteq S^3$. Then if $\Hh$ is a link-minimal Heegaard diagram for the link $(L, \Lambda)$, $\s$ is a spin$^c$-structure on $Y$, and $\D^0$ denotes the length-zero hypercube map for the link-surgery complex $\Cc_{\Lambda}(\Hh)$, then 
                \begin{equation}\label{ls2}
                    H_*(\Cc_{\Lambda}(\Hh, \s), \D^0) \cong \HFh^-(Y; \s)
                \end{equation}
                as $\F[[U]]$-modules.
            \end{theorem}
            Moreover the authors of~\cite{MO} assert (see Remark 9.7 of~\cite{MO}) that when $\s$ is a torsion spin$^c$-structure (as is always the case when $Y$ is a rational homology sphere), it is possible to view~\eqref{ls2} as an isomorphism of absolutely graded groups, and obtain a well-defined grading shift from the isomorphism between the two groups. The authors of~\cite{OSSspec} use this fact in the proof of their main result, to give the link surgery formula a well-defined absolute grading which agrees with the grading on lattice homology (see Proposition~\ref{isom4}). 

            \begin{remark}\label{ls8}
                \emph{For a rational homology sphere, are two ways to view the absolute grading on the link surgery formula. The first is by noting that all generators $\x \in \Cc_{\Lambda}(L)$ are elements of $\Tt_{\al} \cap \Tt_{\be}$, where $(\Sigma, \al, \be)$, is a Heegaard diagram for $S^3$. This means that we can place an absolute grading on these elements, such that the maximal grading of a non-torsion element (in the minus convention) is $0$. The isomorphism map constructed in Theorem~\ref{ls7} is a cobordism map which counts holomorphic triangles in $(\Sigma, \al, \be, \ga)$. The grading shift across the isomorphism of Theorem~\ref{ls7} is therefore precisely as in~\eqref{hee2}.}

                \emph{The other way to view the grading on $\Cc_{\Lambda}(L)$ is to say that since the map from Theorem~\ref{ls7} is an isomorphism on homology, we can use it to assign an absolute grading on $H_*(\Cc_{\Lambda}(\Hh, \s), \D^0)$, according to the rule that the grading of an element $\x \in H_*(\Cc_{\Lambda}(\Hh, \s), \D^0)$ is defined to be equal to the grading of its image in $\HFh^-(Y; \s)$. This is the convention implicitly adopted in~\cite{OSSspec}, and it is only by viewing the gradings on the link surgery complex in this way that Proposition~\ref{isom4}, below, makes sense.}
            \end{remark}

        \section{Lattice homology}\label{lh}

            The goal of this section is to give the definitions from lattice homology sufficient to understand the main result of~\cite{OSSspec} (Theorem 1.1 part 4 and Propositions 4.4 and 4.8 of~\cite{OSSspec}). This will provide enough information to complete the proof of Theorem~\ref{main}.

            \subsection{Definitions}

                Lattice homology is an invariant of negative definite plumbed 3-manifolds originally defined by N\'emethi in~\cite{Nem05}, to give a purely combinatorial means of computing Heegaard Floer homology. For such a 3-manifold $Y$, the invariant $\HFl^*(Y)$ is a hypercube of $\F[U]$-chain complexes equipped with a filtration. In Theorem 8.3 of~\cite{Nem05}, N\'emethi proves that for $Y$ equipped with a particular type of plumbing graph called an\emph{almost rational plumbing}, the zero-th degree of the lattice homology complex $\HFl^0(Y)$ is isomorphic to $\HF^+(Y)$; the result result~\eqref{Nem1} calculating the $d$-invariant for almost rational plumbings follows as a corollary. 

                In~\cite{OSSspec}, the authors use a different version of the lattice homology complex, which is equivalent to N\'emethi's original complex when $Y$ is a negative-definite plumbed 3-manifold. It is this reformulated version we will use here. Let $Y$ be a plumbed 3-manifold, with all notation as in Section~\ref{plu}. In the conventions of~\cite{OSSspec}, the lattice homology of $Y$ is a hypercube of chain complexes $\CFl^-(Y)$ which splits over spin$^c$-structures, as
                \[
                    \CFl^-(Y) = \bigoplus_{\s \in \spin^c(Y)} \CFl^-(Y; \s).
                \]
                The generators of $\CFl(Y)$ over $\F[[U]]$ are pairs $[K, E]$, where $K \in \Char(X)$ and $E \in P(V)$, that is, $E$ is some subset of the set of vertices of the plumbing graph $G$. . The complex has an filtration $\delta$ given by $\delta([K, E]) = |E|$, the number of elements in $E$. One should think of each $E \subseteq V$ as a sublink $M$ of the plumbing link $L$, analogous to the filtration for the link surgery complex above. The generator $[K,E]$ is in spin$^c$-structure $\s$ precisely when $K$ is a representative of $\s$ according to the isomorphism~\eqref{plu1}. 
                
                This complex is equipped with hypercube maps and a differential, but in this paper, we will only be concerned with generators of $\CFl^-(Y)$, so we will not discuss the formal definitions of these maps here. The interested reader may consult~\cite{Nem05} for the original formulation,~\cite{OSSspec} for the reformulation, or Zemke's paper~\cite{Zem21} for a more recent discussion. It is also possible to define other flavors of lattice homology, particularly $\CFl^+$ or $\CFl^{\infty}$, by a construction analogous to the one for Heegaard Floer homology. See e.g. Remark 3.4 of~\cite{OSSspec}, for more details. However, we are only concerned with the ``$-$'' version of lattice homology here. 

                When we are dealing with a torsion spin$^c$-structure $\s$, the lattice complex is also equipped with an absolute $\Q$-grading, which is defined in the following way. First, for a given $I \in P(V)$, define
                \[
                    2 f(K, I) = \sum_{v \in I} K (v) + \sum_{v, v' \in I } v \cdot v',
                \]
                where in the first sum, we are viewing $K$ as an element of $\Hom(H_2(X;\Z), \Z)$, and evaluating on $v \in H_2(X; \Z)$, and in the second, the dot product denotes the  intersection form. Since $K$ is a characteristic vector, $f(K, I)$ is clearly an integer for each $I$. Then we define \emph{the minimal G-weight} of a generator $[K, E]$ of $\CFl^-(Y)$ as
                \[
                    g([K, E]) = \min_{I \subseteq E} f(K, I)
                \]
                Then we define the absolute Maslov grading $\gr: \CFl^-(Y) \to \Q)$ as
                \begin{equation}\label{lh1}
                    \gr(U^i \otimes [K, E]) = -2i + 2 g([K, E]) + |E| + \frac{ K^2 - 2 \chi(W) - 3 \sigma(W)}{4},
                \end{equation}
                where $W$ denotes the cobordism between $S^3$ and $Y$ obtained by cutting out a small $B^4$ from $X$. Note that it is clear by inspection that:
                \begin{lemma}\label{lh2}
                    When $Y$ is a negative-definite plumbed homology sphere with all notation as above, then
                    \[
                        \gr([K, \varnothing]) = \frac{ K^2 - 2 \chi(W) - 3 \sigma(W)}{4} = \frac{K^2 + s}{4}.
                    \]
                \end{lemma}
                
            \subsection{Isomorphisms}\label{isom}

                In this section, we discuss the isomorphisms which will be used in the proof of Theorem~\ref{main}. The result  
                consists of two main propositions, which provide the identifications we will actually use.
                \begin{proposition}\label{isom3}
                    [Proposition 4.4 from~\cite{OSSspec}] There is an isomorphism of chain complexes between the lattice homology complex $\CFl^-(Y)$, and the chain complex with underlying $\F[[U]]$-module $H_*(\Cc_{\Lambda}(\Hh), \D^0)$ and differential $\D^1$.
                \end{proposition}
                \noindent See the discussion surrounding~\eqref{ls6} for the notation used. The second chain complex from Proposition~\ref{isom3} is the $E_1$-term of the spectral sequence of~\cite{OSSspec}. The authors of~\cite{OSSspec} then show that
                \begin{proposition}\label{isom4}
                    [Proposition 4.8 from~\cite{OSSspec}] When $\s$ is a torsion spin$^c$ structure on $Y$, the identification from Proposition~\ref{isom3} respects absolute Maslov gradings. 
                \end{proposition}
                \noindent These two results (together with a portion of the proof of Zemke's result from~\cite{Zem21}) are all that we need to verify Theorem~\ref{main}. 
                
                \begin{remark}\label{isom5}
                    \emph{Propositions~\ref{isom3} and~\ref{isom4} do not prove that the full lattice complex is identified with the full Heegaard Floer complex. The main result of~\cite{OSSspec} provides a spectral sequence, the $E_1$-term of which we considered in the propositions above, rather than an overall isomorphism. The full isomorphism was proved somewhat later by Zemke, in~\cite{Zem21}:}
                    \begin{theorem}\label{isom2}
                        [Theorem 1.1 of~\cite{Zem21}] If $Y$ is a plumbed 3-manifold then there is an isomorphism of $\F[[U]]$-modules
                        \[  
                          \HFl^-(Y) \cong \HFh^{-}(Y)
                        \]
                        respecting the splitting over spin$^c$-structures, and when $b_1(Y) = 0$, this isomorphism is relatively graded.
                    \end{theorem}
                \end{remark}

                It now remains to discuss the portion of Zemke's result, Theorem~\ref{isom2}, which is necessary for the proof of Theorem~\ref{main}. Let all notation be as above. In~\cite{Zem21}, Zemke constructs an auxiliary complex $\tilde{\Cc}_{\Lambda}(L)$, which has the same $\F[[U]]$-generators as the link surgery complex $\Cc_{\Lambda}(L)$, but different operations. He then shows that:
                \begin{proposition}\label{isom6}
                    [Proposition 6.3 of~\cite{Zem21}] The chain complexes $\tilde{\Cc}_{\Lambda}(L)$ is chain homotopy equivalent to $\CFl^-(G)$ over $\F[[U]]$.
                \end{proposition}
                To prove this, Zemke uses the identification from~\cite{OSSspec} (Proposition~\ref{isom3}), which matches generators of $\CFl^-(Y)$ to generators of $H_*(\Cc_{\Lambda}(\Hh), \D^0)$. He then constructs a homotopy equivalence between $H_*(\Cc_{\Lambda}(\Hh), \D^0)$ and the original complex $\tilde{\Cc}_{\Lambda}(L)$. The key facts we need from Zemke's proof of Proposition~\ref{isom6} are the following: 
                \begin{lemma}\label{isom7}
                    \begin{enumerate}[label = (\roman*)]
                        \item The elements of $\Cc_{\Lambda}(L)$ which descend to generators of $\HFh^-(Y)$ on homology (according to the isomorphism of Theorem~\eqref{ls7}) are identified with the elements of $\CFl^-(Y)$ which descend to generators on homology.

                        \item The maps involved in this identification are grading-preserving.
                    \end{enumerate}
                \end{lemma}
                \noindent (Part (ii) actually appears as part of the proof of Theorem~\ref{isom2} for $b_1(Y) = 0$, appearing in Section 6.3 of~\cite{Zem21}; however, it follows from Zemke's proof of Proposition~\ref{isom6}.) 

    \section{Proof of the main result}\label{ver}

        In this section, we give a proof of Theorem~\ref{main}. This proof consists of a detailed analysis of the grading shifts in the isomorphisms from Section~\ref{isom}.

        \begin{proof}
            Let $[k]$ be a spin$^c$ structure of the negative definite plumbed 3-manifold $Y$, which we are assuming is a rational homology sphere. The goal here is to show that the maximal grading of a torsion-free element of $\HFh^-(Y; [k])$ is 
            \begin{equation}\label{ver1}
                \max_{K \in [k]} \frac{K^2 + s}{4}
            \end{equation}
            where all the $K$ are characteristic vectors of $X$, and as usual, $s$ is the number of vertices of the plumbing graph for $Y$. 

            First, look at $\HFl^-(Y; [k])$. There is a single non-torsion module $\F[[U]]$ in $\HFl^-(Y; [k])$, which lives in $\delta$-filtration zero. (See e.g. the discussion following the statement of Theorem 2.8 in~\cite{OSSknot} for a proof of this fact.) This means that non-torsion generators of $\HFl^-(Y,[k])$ will all be of the form $[K, \varnothing]$ where $K \in \Char(X)$ is a representative of $[k]$. By Lemma~\ref{lh2}, the maximal degree of a non-torsion element of $\HFl^-(Y; [k])$ is as in~\eqref{ver1}.

            As noted in Lemma~\ref{isom7}, the combination of Proposition~\ref{isom3} and Proposition~\ref{isom6} gives us an identification between elements of $\C_{\Lambda}(L, [k])$ which descend to generators of $\HFh^-(Y; [k])$ on homology, and elements of $\CFl^-(Y; [k])$ which descend to generators of $\HFh^-(Y; [k])$ on homology. This identification respects spin$^c$ structures and preserves absolute gradings. In particular, the induced map on homology identifies non-torsion elements of $\HFh^-(Y; [k])$ with non-torsion elements of $\HFl^-(Y; [k])$ in a grading preserving way. 

            This means that any maximal degree non-torsion element of $\HFh^-(Y; [k])$ maps to a maximal degree non-torsion element of $\HFl^-(Y; [k])$, and has the same grading. This means that the main result now follows from Lemma~\ref{lh2} and the previous paragraphs.
        \end{proof}

\bibliography{dBib}
\bibliographystyle{plain}

\end{document}